\documentclass[leqno,11pt]{article}
\usepackage[margin=0.8 in  ]{geometry}         
\usepackage{graphicx}
\usepackage{stmaryrd}
\usepackage{framed}
\usepackage{amssymb}
\usepackage{epstopdf}
\usepackage{amsmath,amsfonts,amssymb,color}
\usepackage{enumerate}
\usepackage{multirow}
\usepackage{fixltx2e}
\usepackage{mathpazo}
\usepackage{empheq}

\usepackage{soul}

\definecolor{myblue}{rgb}{.8, .8, 1}
  \newcommand*\mybluebox[1]{%
    \colorbox{myblue}{\hspace{1em}#1\hspace{1em}}}

\usepackage{theorem}

\newcommand{\lev}[1]{\ensuremath{\operatorname{lev}_{#1}}}
\newcommand{\bproj}[1]{\overleftarrow{\thinspace P\thinspace}_%
{\negthinspace\negthinspace #1}}
\newcommand{\scal}[2]{\left\langle{#1},{#2}  \right\rangle}
\newcommand{\menge}[2]{\big\{{#1}~\big |~{#2}\big\}}
\newcommand{\nnn}{\ensuremath{{n\in{\mathbb N}}}}
\newcommand{\Fix}{\ensuremath{\operatorname{Fix}}}
\newcommand{\Id}{\ensuremath{\operatorname{Id}}}

\theoremstyle{plain}{\theorembodyfont{\rmfamily}
\newtheorem{theorem}{Theorem}[section]}
\newtheorem{lemma}[theorem]{Lemma}

\newtheorem{proposition}[theorem]{Proposition}
\theoremstyle{plain}{\theorembodyfont{\rmfamily}
\newtheorem{definition}[theorem]{Definition}}
\theoremstyle{plain}{\theorembodyfont{\rmfamily}
}
\theoremstyle{plain}{\theorembodyfont{\rmfamily}
}
\theoremstyle{plain}{\theorembodyfont{\rmfamily}
\newtheorem{algorithm}[theorem]{Algorithm}}
\theoremstyle{plain}{\theorembodyfont{\rmfamily}
\newtheorem{example}[theorem]{Example}}
\theoremstyle{plain}{\theorembodyfont{\rmfamily}
\newtheorem{fact}[theorem]{Fact}}
\theoremstyle{plain}{\theorembodyfont{\rmfamily}
\newtheorem{remark}[theorem]{Remark}}

 \providecommand{\dom}{\operatorname{dom}}

\providecommand{\RR}{\mathbb{R}}
\providecommand{\RP}{{\mathbb{R}_+}}
\providecommand{\NN}{\mathbb{N}}

\newcommand{\sepp}{\setlength{\itemsep}{-3pt}}

\title{A Bregman projection method for approximating\\ fixed points of quasi-Bregman nonexpansive mappings}

\author{
Heinz H. Bauschke\thanks{Mathematics, Irving K. Barber School,
University of British Columbia,  Kelowna, B.C. V1V 1V7, Canada.
Email: {\tt heinz.bauschke@ubc.ca}. },~
Jiawei Chen\thanks{School of Mathematics and Statistics,
Southwest University, Chongqing 400715, China. Email: 
{\tt J.W.Chen713@163.com}. },
~
and~
Xianfu Wang\thanks{Mathematics, Irving K. Barber School, University of
British Columbia,  Kelowna, B.C. V1V 1V7, Canada. Email: 
{\tt shawn.wang@ubc.ca}.
} }

\date{September 24, 2013}

\begin{document}

\maketitle

\begin{center}
\emph{\large Dedicated to Boris Mordukhovich on the occasion of his
65th Birthday}

\bigskip

\bigskip

\end{center}

\begin{abstract}
\noindent
We introduce an abstract algorithm that aims to find
the Bregman projection onto a closed convex set. As an application,
the asymptotic behaviour of an iterative method for 
finding a fixed point of
 a quasi Bregman nonexpansive mapping with the 
 fixed-point closedness property is analyzed. 
We also show that our result is applicable to 
Bregman subgradient projectors.
\end{abstract}


\section{Introduction}
Bregman distances provide a general and flexible framework for 
studying optimization problems both theoretically and
algorithmically \cite{BB96}--\cite{RS011}.
The objective of this paper is to present an iterative method for 
finding the Bregman projection onto a closed convex set. 
Our results extends those of \cite{BCW} from the Euclidean distance
to the Bregman distance setting. 

The paper is organized as follows. 
Section~\ref{s:assumption} contains useful 
auxiliary results on Bregman distances. 
In Section~\ref{s:iteration}, we introduce and analyze the iteration
scheme for finding the Bregman projection onto a closed convex set.
In Section~\ref{s:asymptotic}, we apply our iteration scheme to quasi 
Bregman nonexpansive
mappings that are fixed-point closed. The iterates are shown to 
converge to the fixed point which is the Bregman nearest point 
to the starting point; moreover, the total length
of the trajectory in terms of the Bregman distance
is finite. We conclude by pointing out that 
our theory applies to Bregman subgradient projectors. 

\section{Assumptions, notions and facts}\label{s:assumption}
\subsection{Standing assumptions}
 We assume throughout this paper that
\begin{empheq}[box=\mybluebox]{equation*}
\text{$C$ is a closed convex subset of a finite dimensional
Euclidean space $X$ }
\end{empheq}
with inner product $\scal{\cdot}{\cdot}$ and norm $\|\cdot\|$, 
and that
\begin{empheq}[box=\mybluebox]{equation*}
\text{$f\colon X\to \RR$ is strictly convex and differentiable,
with $\dom f^*$ open.}
\end{empheq}
We shall assume that the reader is familiar with basic convex
and variational analysis and its notation; see, e.g.,
\cite{Mord}, \cite{R70}, \cite{Zalines}, or \cite{BC11}. 
Our assumptions imply that $f$ is a Legendre function as hence is it
Fenchel conjugate $f^*$. 

\subsection{Bregman distance and projection}

\begin{definition}[Bregman distance]
\label{def:11}
{\rm (See \cite{B67}.)}
The function 
\begin{equation*}
D \colon X\times X \to \RP\colon 
(x,y)\mapsto f(x)-f(y)-\scal{\nabla f(y)}{x-y}
\end{equation*}
is called the \emph{Bregman distance} with respect to $f$.
\end{definition}
It is well known (see, e.g., \cite{BB97} or \cite{CZ}) that 
the Bregman distance allows nonorthogonal projections in our
setting: 
\begin{definition}[(left) Bregman projection]
For every $y \in X$, there exists a unique point
$\bproj{C}(y)$ in $C$, 
called the \emph{(left) Bregman projection} of $y$ onto $C$, 
such that 
$D(\bproj{C}(y),y) = \min_{c\in C}D(c,y)$. 
\end{definition}
Note that when 
$f=(1/2)\|\cdot\|^2$, then $\bproj{C}$ is the classical orthogonal projector.

\subsection{Useful facts}

The following results, which are mostly well known and which will
be useful later,
are recalled here for the reader's convenience. 

\begin{fact}\label{fact:13}
{\rm (See, e.g., \cite[Fact~2.3]{BC3}.)} 
For every $x \in X$,
the projection 
$\bproj{C}(x)$ 
 is characterized by
\begin{equation}\label{e:leftproj}
\bproj{C}(x)\in C\quad\text{and}\quad 
(\forall c\in C)\quad\langle \nabla f(x)-\nabla f(\bproj{C}(x))  , 
c-\bproj{C}(x) \rangle \leq 0;
\end{equation}
equivalently, by 
\begin{equation}\label{e:lprojchar}
\bproj{C}(x)\in C\quad\text{and}\quad 
(\forall c\in C)\quad D(c,x)\geq D(c,\bproj{C}(x)) +D_{f}(\bproj{C}(x),x).
\end{equation}
Moreover, $\bproj{C}\colon X\to C$ is continuous. 
\end{fact}

\begin{lemma}\label{le:16}
Let $(x_n)_\nnn$ be a sequence in $X$,
let $(y_n)_\nnn$ be a bounded sequence in $X$,
and suppose that $D(x_n,y_n)\to 0$.
Then $x_n-y_n\to 0$. 
\end{lemma}
\begin{proof}
Combine \cite[Remark~2.14]{BR} with \cite[Theorem~2.10]{BR}.
\end{proof}

\begin{lemma}
\label{le:19} 
Let $(x_n)_\nnn$ be a sequence in $X$.
Then the following are equivalent:
\begin{enumerate}
\item 
\label{le:19.1} 
$(x_n)_\nnn$ is bounded.
\item \label{le:19.3}  
$(D(x_{n},y))_\nnn$ is bounded for every $y\in X$. 
\item \label{le:19.2}  
There exists $y\in X$ such that 
$(D(x_{n},y))_\nnn$ is bounded.
\end{enumerate}
\end{lemma}

\begin{proof}  
``\ref{le:19.1}$\Rightarrow$\ref{le:19.3}'':
Let $y\in X$ and suppose that $(D(x_n,y))_\nnn$ is not bounded.
After passing to a subsequence if necessary, we assume that
$x_n\to{x}\in X$ and that $D(x_n,y)\to+\infty$. 
On the other hand, $D(x_n,y) \to D(x,y)\in\RP$.
Altogether, we have reached a contradiction. 

``\ref{le:19.3}$\Rightarrow$\ref{le:19.2}'': This is clear. 

``\ref{le:19.2}$\Rightarrow$\ref{le:19.1}'': 
Suppose that $(x_n)_\nnn$ is not bounded. 
After passing to a subsequence if necessary, we assume that
$\|x_n\|\to +\infty$. 
By \cite[Theorem~3.7.(iii)]{BB97}, $D(\cdot,y)$ is coercive and
thus $D(x_n,y)\to+\infty$, which is absurd.
\end{proof}

\begin{lemma}\label{le:111}
Let $x\in X$ and let $(y_n)_\nnn$ be a sequence in $X$ such that 
$(D(x,y_n))_\nnn$ is bounded. Then $(y_n)_\nnn$ is bounded.
 \end{lemma}

\begin{proof}
This follows readily from \cite[Corollary~3.11]{BB97}. 
\end{proof}

\vskip3mm

\section{Finding the Bregman projection by iteration}

\label{s:iteration}

In this section, we present an iteration scheme to find the 
projection
$\bproj{C}(x_0)$. 
It will be convenient to set, 
for every $(x,y)\in X\times X$, 
\begin{empheq}[box=\mybluebox]{align*}
H(x,y) &:= \menge{z\in X}{D(z,y)\leq D(z,x)}\\
&=\menge{z\in X}{\scal{\nabla f(x)-\nabla f(y)}{z}\leq f(y)-f(x)+
\scal{\nabla f(x)}{x}-\scal{\nabla f(y)}{y}},
\end{empheq}
which is either equal to $X$ (if $x=y$) or to a closed
halfspace (if $x\neq y$). 

\begin{algorithm}\label{alg:I} 
Given $x_0\in X$ and a nonempty closed convex subset $C_0$ of
$X$, set $n := 0$. 

Step 1. Take $y_{n}\in  X$ and set $C_{n+1}:=C_{n}\cap H(x_{n}, y_{n})$.\vskip1mm

Step 2. Compute
\begin{align}\label{alg:theprojection}
x_{n+1}:=\bproj{C_{n+1}}(x_{0})
\end{align}
and stop if provided a stopping criterion is satisfied.

Step 3. Set $n:=n+1$ and go to Step 1.
\end{algorithm}
\vskip2mm

\begin{remark}\label{re:alg}
Since each $H(x_n,y_n)$ is equal to either $X$ or some closed halfspace,
we note that the each set $C_{n+1}$ is closed and convex;
furthermore, if $C_0$ is a polyhedron, then so is $C_{n+1}$. 
\end{remark}

Let us collect some basic properties of Algorithm~\ref{alg:I}.

\begin{proposition}
\label{prop1}
Suppose that $(x_n)_\nnn$ is a sequence generated by
Algorithm~\ref{alg:I}.
Then the following hold:
\begin{enumerate}
\item \label{prop1.1}
{\bf (decreasing sets)} 
$(\forall\nnn)$ 
$C_n \supseteq C_{n+1}$.
\item \label{prop1.2}
{\bf (increasing distances)}
$(\forall\nnn)$ $D(x_{n},x_{0})\leq D(x_{n+1},x_{0})$. 
\item \label{prop1.3}
$(\forall k\in\NN)$ 
$\displaystyle \sum_{n=0}^{k} D(x_{n+1},x_{n})\leq D(\bproj{C_{k+1}}x_{0},x_{0})$.
\item \label{prop1.4}
The constant 
\begin{equation}\label{eq:23}
\beta := \lim_{\nnn}D(x_{n},x_{0}) = \sup_{n\in \NN}D(x_{n},x_{0})
\end{equation}
is well defined.

\item \label{prop1.5}
For all nonnegative integers $m$ and $n$ such that $m<n$, we have 
\begin{equation}\label{eq:24}
\langle \nabla f(x_0)-\nabla f(x_m), x_n-x_m\rangle \leq 0
\end{equation}
and
\begin{equation}\label{eq:25}
D(x_n,y_m)\leq D(x_n,x_m).
\end{equation}
\end{enumerate}
\end{proposition}
\begin{proof}
We only show \ref{prop1.3} and \ref{prop1.5} because the
other properties are clear.

\ref{prop1.3}: In \eqref{e:lprojchar}, put $x=x_{0}$, $C=C_{n}$. For
$x_{n+1}\in C_{n+1}\subseteq C_{n}$, we have
$D(x_{n+1},x_0)\geq D(x_{n+1},x_{n})+D(x_{n},x_{0})$, i.e., 
$D(x_{n+1},x_0)-D(x_{n},x_{0})\geq D(x_{n+1},x_{n})$.
Now sum the last inequality over $n\in\{0,1,\ldots,k\}$. 

\ref{prop1.5}: In \eqref{e:leftproj}, put $x=x_{0}$, $c=x_{n}$, $C=C_{m}$, noting that
$x_{m}=\overleftarrow{P}_{C_{m}}x_{0}$ and
$x_{n}\in C_{n}\subseteq C_{m}$ when $n>m$. This gives \eqref{eq:24}.
Finally \eqref{eq:25} follows because 
$x_{n}\in C_{n}\subseteq C_{m+1}=C_{m}\cap H(x_{m},y_{m})$.
\end{proof}

\bigskip

We now begin the convergence analysis of Algorithm~\ref{alg:I}.

\begin{lemma}
\label{le:21} 

Suppose that $(x_n)_\nnn$ is a sequence generated by
Algorithm~\ref{alg:I} and that $(x_n)_\nnn$ is bounded.
Then the following hold:
\begin{enumerate}
\item 
\label{le:21.0} 
$\sum_{\nnn} D(x_{n+1},x_{n})<+\infty$.
\item 
\label{le:21.1}
$x_{n+1}-x_{n}\to 0$. 
\end{enumerate}
\end{lemma}
\begin{proof}
\ref{le:21.0}:
By Lemma~\ref{le:19}, 
$(D(\bproj{C_{n}}x_{0},x_{0}))_{\nnn}
= (D(x_{n},x_{0}))_{\nnn}$ is bounded. 
Now apply Proposition~\ref{prop1}\ref{prop1.3}.

\ref{le:21.1}: 
Since $D(x_{n+1},x_{n})\to 0$ by \ref{le:21.0},
we deduce 
from Lemma~\ref{le:16} that $x_{n+1}-x_{n}\to 0$.
\end{proof}

\begin{lemma}\label{le:22}
Suppose that $(x_n)_{n\in \NN}$ is a sequence generated by
Algorithm~\ref{alg:I} such that 
for every subsequence $(x_{k_n})_{n\in \NN}$ of $(x_n)_{n\in \NN}$, we have
\begin{equation}\label{eq:26}
\left.
\begin{array}{c}
x_{k_n}\to \bar{x}\\
x_{k_{n}}-y_{k_n}\to 0
\end{array}
\right\}
\;\;\Rightarrow\;\;
\bar{x}\in C.
\end{equation}
Then every bounded subsequence of $(x_n)_{n\in \NN}$ must converge 
to a point in $C$.
\end{lemma}
\begin{proof}
Suppose that $(x_{k_n})_{n\in \NN}$ is a bounded subsequence of 
$(x_n)_{n\in \NN}$. 
By Lemma~\ref{le:19}, $(D(x_{k_n},x_0))_\nnn$ is bounded.
Hence Proposition~\ref{prop1}\ref{prop1.4} implies that the constant
$\beta$ defined in \eqref{eq:23} belongs to $\RP$. 
Let $m$ and $n$ be in $\NN$ such that $m<n$. 
Then $x_{k_n}\in C_{k_n}\subseteq C_{k_m}$.
Using Fact \ref{fact:13} (applied to $x=x_{0},C=C_{k_m}$) 
and \eqref{eq:23}, we have
\begin{equation*}
D(x_{k_n},x_{k_m})
\leq D(x_{k_n},x_0) -D(x_{k_m},x_0)\to \beta-\beta = 0
\quad \text{as}\quad n>m\to +\infty.
\end{equation*}
It now follows from Lemma~\ref{le:16} that
$x_{k_n}-x_{k_m}\to 0$
as $n>m\to +\infty$,
i.e., $(x_{k_n})_{n\in \NN}$ is a Cauchy sequence.
Therefore, 
\begin{equation}
\label{e:0922a}
x_{k_{n+1}}-x_{k_n}\to 0
\end{equation} and
there exists $\bar{x}\in X$ such that 
$$x_{k_n}\to \bar{x}.$$
It follows from Remark~\ref{re:alg} and \eqref{alg:theprojection} 
that $\bar{x}\in C_{k_n}$ and that $\bar{x}\in H(x_{k_n},y_{k_n})$ 
for every $n\in \NN$.
By the definition of $H$, one has
$$D(\bar{x},y_{k_n}) \leq D(\bar{x},x_{k_n}) =f(\bar{x}) -f
(x_{k_n}) - \langle \nabla f (x_{k_n}),  \bar{x}- x_{k_n}\rangle
\to 0.$$
Hence $(D(\bar{x},y_{k_n}))_{n\in \NN}$ is bounded.
By Lemma~\ref{le:111}, $(y_{k_n})_{n\in \NN}$ is bounded too.
Now, from $x_{k_{n+1}}\in C_{k_{n+1}}\subseteq H(x_{k_n},y_{k_n})$, 
we obtain
$$D(x_{k_{n+1}},y_{k_n}) \leq D(x_{k_{n+1}},x_{k_n})  \to 0.$$
Again from Lemma~\ref{le:16}, one has
\begin{equation}
\label{e:0922b}
x_{k_{n+1}}-y_{k_n}\to 0.
\end{equation}
Combining \eqref{e:0922a} with \eqref{e:0922b}, we deduce that 
 $$\|x_{k_{n}}-y_{k_n}\|\leq \|x_{k_{n}}-x_{k_{n+1}}  \|+ \|x_{k_{n+1}}-y_{k_n}\|\rightarrow 0.$$
This and \eqref{eq:26} yield the result. 
\end{proof}

\bigskip

Lemma~\ref{le:21} and Lemma~\ref{le:22} allow us to derive the following dichotomy result.

\begin{theorem}[dichotomy] \label{th:23}
Suppose that $(x_n)_{n\in \NN}$ is generated by
Algorithm~\ref{alg:I}, that $(\forall \nnn)$ $C\subseteq C_n$,
and that for every subsequence $(x_{k_n})_{n\in \NN}$ of $(x_n)_{n\in \NN}$, we have
\begin{equation}\label{eq:28}
\left.
\begin{array}{c}
x_{k_n}\to \bar{x}\\
x_{k_{n}}-y_{k_n}\to 0
\end{array}
\right\}
\;\;\Rightarrow\;\;
\bar{x}\in C.
\end{equation}
Then exactly one of the following holds:
\begin{enumerate}
\item 
\label{th:23.0}
$C\neq\varnothing$, $x_n\to \bproj{C}x_0$,
and
$\sum_{\nnn}D(x_{n+1},x_{n})<+\infty$.
\item \label{th:23.1}
$C=\varnothing$ and $\|x_n\|\to+\infty$.
\end{enumerate}
\end{theorem}
\begin{proof}
Note first that
\begin{equation}\label{eq:29}
(\forall n\in \NN)\quad
D(x_{n},x_0) =\inf_{c\in C_{n}}D(c,x_0)\leq \inf_{c\in C}D(c,x_0).
\end{equation}
\ref{th:23.0}:
Assume that $C\neq\varnothing$.
Then $(x_n)_{n\in \NN}$ is bounded by \eqref{eq:29} and Lemma~\ref{le:19}.
By Lemma~\ref{le:22},
 $$\bar{x} := \lim_{n\in \NN} x_n \in C.$$
On the other hand, \eqref{eq:29} yields
\begin{equation*}
D(\bar{x},x_0) \leq \inf_{c\in C}D(c,x_0).
\end{equation*}
Altogether, $\bar{x}=\bproj{C}x_0$. 
Finally, $\sum_{\nnn} D(x_{n+1},x_n)<+\infty$ because of 
Lemma~\ref{le:21}.

\ref{th:23.1}:
Suppose that $\|x_n\|\not\rightarrow+\infty$.
Then $(x_n)_{n\in \NN}$ contains a bounded subsequence which,
by Lemma \ref{le:22}, must converge to a point in $C$.
Therefore if $C=\varnothing$, then $\|x_n\|\rightarrow+\infty$.
\end{proof}

\section{Fixed points of quasi Bregman nonexpansive mappings}
\label{s:asymptotic}

In this section, we shall apply the results in Section~\ref{s:iteration} to find the Bregman nearest fixed point of a quasi Bregman nonexpansive mapping.

\subsection{Quasi Bregman nonexpansive (QBNE) mappings}
Let $E$ be a nonempty closed convex subset of $X$. 
The \emph{fixed point set} of $T:E\to X$ is
$\Fix T :=\menge{x\in E}{Tx=x}$.

\begin{definition}\label{def:30}
Let $E$ be a nonempty closed convex subset of $X$,
and let $T\colon E\to X$. Then $T$ is said to be:
\begin{enumerate}
\item \emph{fixed-point closed} if, for every sequence
$(x_{n})_\nnn$ in $E$, 
$\displaystyle \left.
\begin{array}{c}
x_{n}\to \bar{x}\\
x_{n}-Tx_{n}\to 0
\end{array}
\right\}
\;\;\Rightarrow\;\;
\bar{x}\in \Fix T$.
\item \emph{quasi Bregman nonexpansive (QBNE)} if
$(\forall x\in \Fix T)(\forall y\in E)$
$D(x,Ty)\leq D(x,y)$.
\end{enumerate}
\end{definition}

It is easy to see that if $T:E\to X$ is QBNE, then
$\Fix T\subseteq \bigcap_{x\in E}H(x,Tx)$.

\begin{fact}\label{fact:34}
Let $E$ be a nonempty closed convex subset of $X$,
and let $T\colon E\to X$ be QBNE.
Then $\Fix T$ is closed and convex.
\end{fact}
\begin{proof}
Inspect the 
\cite[proof of Lemma~15.5]{RS011}, or combine
\cite[Proposition~3.3(iv)\&(vii)]{BBC03}.
\end{proof}

\subsection{Finding the  Bregman nearest fixed point}

When applied to a quasi Bregman nonexpansive mapping with the fixed-point closedness property, Algorithm~\ref{alg:I}
and Theorem~\ref{th:23} together
provide an iterative method for finding the
Bregman nearest fixed point.

\begin{theorem}[trichotomy] \label{th:35}
Let $E$ be a nonempty closed convex subset of $X$,
let $T\colon E\to X$ be QBNE and fixed-point closed, 
let $x_0\in X$, and let $C_0$ be a closed convex nonempty subset
of $X$ containing $\Fix T$.
Define sequences $(C_n)_\nnn$ and $(x_n)_\nnn$ by 
\begin{equation*}
\label{eq:31}
(\forall \nnn)\quad
C_{n+1} := C_n \cap H(x_n,Tx_n)
\;\;\text{and}\;\;
x_{n+1} = \bproj{C_{n+1}}x_0.
\end{equation*}
Then exactly one of the following holds:
\begin{enumerate}
\item\label{th:35.0}
$\Fix T\neq\varnothing$, $x_n\to\bproj{\Fix T}x_0$
and
$\sum_{n\in\NN}D(x_{n+1},x_{n})<+\infty$. 
\item
\label{th:35.1}
$\Fix T=\varnothing$ and $\|x_n\|\to+\infty$.
\item
\label{th:35.2} 
$\Fix T=\varnothing$ and the sequence is not well defined 
(i.e.,
$C_{n+1}=\varnothing$ for some $n\in\NN$).
\end{enumerate}
\end{theorem}
\begin{proof}
Suppose that $C=\Fix T$ and set 
$(y_{n})_\nnn=(Tx_{n})_\nnn$ when $(x_n)_\nnn$
is well defined.
In this case, it is clear that \eqref{eq:28} holds because $T$ is
fixed-point closed.

\ref{th:35.0}: Assume that $C\neq\varnothing$. We show inductively
that $(\forall\nnn)$ $C\subseteq C_n$. Note that $C\subseteq
C_0 \neq\varnothing$. Suppose that $C\subseteq C_n$ for
some $n\in \NN$. Then $x_n$ is well defined
and $C\subseteq H(x_n,Tx_n)$ because $T$ is QBNE.
Moreover, $C\subseteq  C_n\cap H(x_n,Tx_n)
= C_{n+1}$. Therefore $(\forall \nnn)$ 
$C\subseteq C_n$, and $C_n$ is nonempty, closed, and convex 
by Remark~\ref{re:alg}.
Hence, the sequence $(x_n)_\nnn$ is well defined.
The conclusion thus follows from Theorem~\ref{th:23}.

\ref{th:35.1}$\&$\ref{th:35.2}: Assume that
$C=\varnothing$.
If $(x_n)_\nnn$ is not well defined, then \ref{th:35.2} happens.
Finally, if $(x_n)_\nnn$ is well defined, then \ref{th:35.1}
occurs, again by Theorem~\ref{th:23}.
\end{proof}

\subsection{Bregman subgradient projectors}

Let us now show that every Bregman subgradient projector is QBNE and 
that it has the fixed point closedness property.
We can also arrange that $C$ is its fixed point set. 
This guarantees that Theorem~\ref{th:35} is applicable to
Bregman subgradient projectors.

For the remainder of this paper, we assume that
\begin{empheq}[box=\mybluebox]{equation*}
\text{$g:X\to \RR$ is a continuous and convex with }
\lev{\leq 0} g:=\menge{x\in X}{g(x)\leq 0} \neq \varnothing,
\end{empheq}
and that 
\begin{empheq}[box=\mybluebox]{equation*}
(\forall z\in X)(\forall z^*\in\partial g(z))\quad
H_{g}(z,z^*):=\menge{ x \in X}{ g(z) +\scal{z^{*}}{ x-z}\leq 0}.
\end{empheq}

The following result follows directly from the definitions.

\begin{proposition}
\label{prop:40} 
Let $z\in X$ and let $z^*\in\partial g(z)$. 
Then the following hold:
\begin{enumerate}
\item 
\label{prop:40.0} 
$\lev{\leq 0} g\subseteq H_{g}(z,z^{*})$.
\item 
\label{prop:40.1} 
$H_{g}(z,z^*)$ is convex, closed, and nonempty; it is a halfspace
when $z^*\neq 0$.
\item 
\label{prop:40.2} 
$z\in H_{g}(z,z^{*})\Leftrightarrow z\in\lev{\leq 0} g$.
\end{enumerate}
\end{proposition}

\begin{definition}
Let $s\colon X\to X$ be a selection of $\partial g$, i.e.,
$(\forall z\in X)$ $s(z)\in\partial g(z)$. 
The associated \emph{(left) Bregman subgradient projector} onto
$\lev{\leq 0}g$ is 
\begin{equation}\label{eq:41.0}
Q_s \colon X\to X \colon z\mapsto \bproj{H_g(z,s(z))}(z).
\end{equation}
\end{definition}

The following result is known.

\begin{lemma}
\label{le:41}
{\rm (See \cite[Propositions~3.3 and 3.38]{BBC03}.)}
$Q_s$ is QBNE with $\Fix Q_{s} =\lev{\leq 0} g$.
\end{lemma}

We now show that $Q_{s}$ is fixed-point closed.

\begin{lemma}
\label{le:42}
$Q_s$ is fixed-point closed.
\end{lemma}
\begin{proof}
Let $(x_n)_\nnn$  be a sequence in $X$ such that 
$x_n\to \bar{x}$
and
\begin{equation}
\label{eq:42}
x_n-Q_{s}(x_n) \to 0.
\end{equation}
We must show that $\bar{x}\in \Fix Q_{s}$. 
In view Lemma~\ref{le:41}, it suffices to show that
$g(\bar{x})\leq 0$. 

Set $(\forall\nnn)$ $p_{n}:=Q_{s}(x_n)$. 
For every $\nnn$, by the definition of $Q_{s}(x_n)$,
$p_{n}$ minimizes the function $$y\mapsto D(y,x_{n})=f(y)-f(x_{n})-\scal{\nabla f(x_{n})}{y-x_{n}}$$ 
over the set
$H_{g}(x_{n},s(x_n))
=\menge{x\in X}{g(x_{n})+\scal{s(x_{n})}{x-x_{n}}\leq 0}$, 
where $s(x_{n})\in \partial g(x_{n})$; hence 
\begin{equation}
\label{e:0922c}
g(x_{n})+\scal{s(x_{n})}{p_{n}-x_{n}}\leq 0.
\end{equation}
Since $x_{n}\to \bar{x}$, it follows that $g(x_n)\to g(\bar{x})$
and $(s(x_{n}))_\nnn$ is bounded.
It therefore follows from \eqref{eq:42} and \eqref{e:0922c} that 
$g(\bar{x})\leq 0$, as required.
\end{proof}

\bigskip

Combining Theorem~\ref{th:35}, Lemma~\ref{le:41}, and
Lemma~\ref{le:42},
we obtain the following result.

\begin{theorem}
Let $x_{0}\in X$,
and let $C_0$ be a closed convex subset of $X$ such that
$\lev{\leq 0}g\subseteq C_{0}$. 
Define sequence $(x_n)_\nnn$ and $(C_n)_\nnn$ via 
$$(\forall\nnn) \quad 
C_{n+1}:=C_{n}\cap H(x_{n},Q_{s}x_{n}) \text{ and }
x_{n+1}:= \bproj{C_{n+1}}x_0.
$$
Then $x_n\to \bproj{\lev{\leq 0}g}x_0$ and
$\sum_{\nnn}D(x_{n+1},x_{n})<+\infty$. 
\end{theorem}

We conclude with a few examples of Bregman subgradient projectors
illustrating that this class is quite large.

\begin{example}\label{ex:45}
Suppose that $f=\frac{1}{2}\|\cdot\|^{2}$ and that 
$g$ is differentiable on $X\smallsetminus \lev{\leq 0} g$.
The Bregman subgradient projector (see \eqref{eq:41.0})
then turns into the classical subgradient projector 
\begin{equation}\label{e:subproj}
Q\colon X\to X\colon x\mapsto \begin{cases}
x, &\text{if $g(x)\leq 0$;}\\
x-\frac{g(x)}{\|\nabla g(x)\|^{2}} \nabla g(x), &\text{otherwise.}
\end{cases}
\end{equation}
By Lemma~\ref{le:41} and \ref{le:42},
$Q$ is QBNE and fixed-point closed. 
We single out two special cases:
\begin{enumerate}
\item 
Suppose $g = \tfrac{1}{p}d_C^p$, where
$1\leq p<+\infty$ and 
$(\forall x\in X)$ $d_C(x) := \min_{c\in C}\|x-c\|$.
Then 
\begin{equation}
\label{e:0922d}
Q=\big(1-\tfrac{1}{p}\big)\Id +\tfrac{1}{p}P_{C},
\end{equation}
where $\Id := P_X$. 
Indeed, if $x\not\in C$, then
$\nabla g(x)=d_{C}^{p-2}(x)(x-P_{C}(x))$ 
and \eqref{e:0922d} follows from \eqref{e:subproj}.

\item Suppose $g= e_h$, 
where $h:X\to \RR$ is convex, lower semicontinuous, proper, with
$\lev{\leq 0} h\neq \varnothing$ and where
$e_h$ is the Moreau envelope of $h$, i.e.,
\begin{align}\label{eq:45}
(\forall x\in X)\quad 
e_{h}(x):=\inf_{w\in X} \big(h(w)+\tfrac{1}{2} \|w-x\|^{2}\big).
\end{align}
Then
\begin{equation}\label{eq:46}
Q\colon X\to X\colon x\mapsto \begin{cases}
x, &\text{if $e_{h}(x)\leq 0$;}\\[+2mm]
\displaystyle \frac{e_{h}(x)-2h(P_{h}(x))}{2(e_{h}(x)-h(P_{h}(x)))}x+  \frac{e_{h}(x)}{2(e_{h}(x)-h(P_{h}(x)))}P_{h}(x),
& \text{ otherwise,}
\end{cases}
\end{equation}
where $P_{h}(x):=\mbox{argmin}_{w\in X}(h(w)+\frac{1}{2} \|w-x\|^{2} )$ 
denotes the proximal mapping of $h$.
To see \eqref{eq:46}, we start by observing that 
$\lev{\leq 0} h\subseteq \lev{\leq 0} e_h\neq \varnothing$
because $\lev{\leq 0} h\neq \varnothing$ and $e_{h}\leq h$. 
By e.g.\ \cite[Theorem~2.26]{RW}, 
$P_{h}$ is single-valued and continuous, 
and $e_{h}$ is convex and continuously differentiable with
$\nabla e_{h}=(\Id-P_{h})$.
From \eqref{eq:45} it follows that
\begin{align}\label{eq:47}
 e_{h}(x)=h(P_{h}(x))+\frac{1}{2}\|x-P_{h}(x)\|^{2};
 \end{align}
thus, 
\begin{align}\label{eq:48}
\|x-P_{h}(x)\|^{2}=2(e_{h}(x)-h(P_{h}(x))).
\end{align}
Combining
\eqref{eq:47}, \eqref{eq:48}, and \eqref{e:subproj},
we obtain \eqref{eq:46}.
Note that
$e_{h}(x)\leq 0$ $\Rightarrow$ $h(P_{h}(x))\leq 0$.
\end{enumerate}
\end{example}

\section*{Acknowledgments}

HHB was partially supported by the Natural Sciences and
Engineering Research Council of Canada and by the Canada Research Chair
Program.
JC was partially supported by the Natural Science Foundation of China, the Doctor Fund of Southwest University and the Fundamental Research
Fund for the Central Universities.
XW was partially supported by the Natural
Sciences and Engineering Research Council of Canada.

\bibliographystyle{plain}

\begin{thebibliography}{100}
\sepp


\bibitem{BB96} H.H.\ Bauschke and J.M.\ Borwein, 
On projection algorithms for solving convex feasibility problems, 
\emph{SIAM Rev.}~38 (1996), 367--426.

\bibitem{BB97} H.H.\ Bauschke and J.M.\ Borwein, Legendre functions
and the method of random Bregman projections, \emph{J. Convex
Anal.}~4 (1997), 27--67.

\bibitem{BBC} H.H.\ Bauschke, J.M.\ Borwein, and P.L.\ Combettes, Essential smoothness, essential strict convexity, and
 Legendre functions in Banach spaces, 
 \emph{Commun. Contemp. Math.}~3 (2001), 615--647.

\bibitem{BBC03} H.H.\ Bauschke, J.M.\ Borwein, and P.L.\ Combettes, 
Bregman monotone optimization algorithms, 
\emph{SIAM J. Control Optim.}~42 (2003), 596--636.

\bibitem{BCW} H.H.\ Bauschke, J.\ Chen, and X.\ Wang, 
A projection method for approximating
fixed points of quasi nonexpansive mappings without
the usual demiclosedness condition, \emph{J. Nonlinear Convex Anal.}, to appear.


\bibitem{BC03} H.H.\ Bauschke and P.L.\ Combettes, Construction of best Bregman approximation in reflexive Banach spaces,
\emph{Proc. Amer. Math. Soc.}~131 (2003), 3757--3766.

\bibitem{BC3} H.H.\ Bauschke and P.L.\ Combettes, Iterating
Bregman retraction, \emph{SIAM J. Optim.}~13 (2003), 1159--1173.

\bibitem{BC11} H.H.\ Bauschke and P.L.\ Combettes, \emph{Convex Analysis and Monotone Operator Theory in Hilbert Spaces}, Springer, New York, 2011.

\bibitem{BN} H.H.\ Bauschke and D.\ Noll, The method of forward
projection, \emph{J. Nonlinear Convex Anal.}~3 (2002), 191--205.


\bibitem{B67} L.M.\ Bregman, The relaxation method for finding
common points of convex sets and its application to the solution of
problems in convex programming, \emph{USSR Comput. Math.
Math. Phys.}~7 (1967), 200--217.


\bibitem{BI} D.\ Butnariu and A.N.\ Iusem, \emph{Totally Convex Functions
for Fixed Points Computation and Infinite Dimensional Optimization,
Applied Optimization}, vol. 40, Kluwer Academic, Dordrecht, 2000.

\bibitem{BR} D.\ Butnariu and E.\ Resmerita, Bregman distances,
totally convex functions, and a method for solving operator
equations in Banach spaces, \emph{Abstr. Appl. Anal.}~84919
(2006), 39 pages.


\bibitem{CR98} Y. Censor and S. Reich,
The Dykstra algorithm with Bregman projections,
\emph{Commun. Appl. Anal.} 2 (1998), 407--419.

\bibitem{CZ}
Y. Censor and S.A.\ Zenios,
\emph{Parallel Optimization},
Oxford University Press, New York, 1997. 

\bibitem{KRS} G.\ Kassay, S.\ Reich, and S.\ Sabach, Iterative
methods for solving systems of variational inequalities in
reflexive Banach spaces, \emph{SIAM J. Optim.}~21 (2011), 1319--1344.

\bibitem{Mord}
B.S.\ Mordukhovich,
\emph{Variational Analysis and Generalized Differentiation I},
Springer, New York, 2006. 

\bibitem{RS10} S.\ Reich and S.\ Sabach, A projection method for solving nonlinear
problems in reflexive Banach spaces, \emph{J. Fixed Point Theory
Appl.}~9 (2011), 101--116.

\bibitem{RS011} S.\ Reich and S.\ Sabach, 
Existence and approximation of fixed points of Bregman firmly
nonexpansive operators in reflexive Banach spaces, in: 
\emph{Fixed-Point
Algorithms for Inverse Problems in Science and Engineering,
Optimization and its Applications}, Springer, New York,
2011, pp. 301--316.

\bibitem{R70}R.T.\ Rockafellar, \emph{Convex Analysis}, Princeton University Press, Princeton, NJ, 1970.

\bibitem{RW}R.T.\ Rockafellar and R.J-B Wets, 
\emph{Variational Analysis}, 
corrected 3rd printing,
Springer, Berlin, 2009.

\bibitem{Zalines}
C.\ Z\u{a}linescu, \emph{Convex Analysis in General Vector Spaces}, World Scientific Publishing, 2002.

\end{thebibliography}

\end{document}